\documentclass[unknownkeysallowed, 10pt, a4]{amsart}
\usepackage[utf8]{inputenc}
\usepackage[margin=1in]{geometry}
\usepackage[linguistics]{forest}
\usepackage{graphicx}
\usepackage{amsaddr}
\usepackage{mathtools}
\usepackage{array}
\usepackage{amsmath}
\usepackage{enumitem}
\usepackage{amsfonts}
\usepackage{amsthm}
\usepackage{amssymb,latexsym,hyperref}
\usepackage{xcolor}
\usepackage{tikz-cd}
\usepackage{tikz}
\usetikzlibrary{patterns}

\def\tred{\textcolor{red}}

\newcolumntype{L}{>{$}l<{$}}
\def\defeq{\coloneqq}
\def\eqdef{=:}

\title[CSB theorem for countable shuffles]{Cantor-Schr\"oder-Bernstein theorem for a class of countable linear orders\vspace{-6mm}}
\author[Srivastava and Mittal]{Suyash Srivastava and Mihir Mittal}
\address{Department of Mathematics and Statistics\\Indian Institute of Technology, Kanpur\\ Uttar Pradesh, India\\Corresponding author: Suyash Srivastava (ORCID ID: 0009-0001-0269-7361)}
\email{suyashsriv20@gmail.com, mihirmittal24@gmail.com}

\begin{document}

\newtheorem{defn}{Definition}

\newtheorem{manualthm}{Theorem}
\setlength{\fboxsep}{1pt}
\setlength{\fboxrule}{1pt}
\newtheorem{innercustomthm}{Theorem}
\newenvironment{customthm}[1]
  {\renewcommand\theinnercustomthm{\fbox{#1}}\innercustomthm}
  {\endinnercustomthm}
  \newtheorem{innercustomconsthm}{Construction and Theorem}
\newenvironment{customconsthm}[1]
  {\renewcommand\theinnercustomconsthm{\fbox{#1}}\innercustomconsthm}
  {\endinnercustomconsthm}
  
\newtheorem{innercustomlem}{Lemma}
\newenvironment{customlem}[1]
  {\renewcommand\theinnercustomlem{\fbox{#1}}\innercustomlem}
  {\endinnercustomlem}
  
\newtheorem{innercustomcons}{Construction}
\newenvironment{customcons}[1]
  {\renewcommand\theinnercustomcons{\fbox{#1}}\innercustomcons}
  {\endinnercustomcons}
  \newtheorem{innercustomcor}{Corollary}
\newenvironment{customcor}[1]
  {\renewcommand\theinnercustomcor{\fbox{#1}}\innercustomcor}
  {\endinnercustomcor}
  
\newtheorem{innercustomdefn}{Definition}
\newenvironment{customdefn}[1]
  {\renewcommand\theinnercustomdefn{\fbox{#1}}\innercustomdefn}
  {\endinnercustomdefn}
  
\newtheorem{definitions}[defn]{Definitions}

\newtheorem{thm}{Theorem}
\newtheorem{lem}[thm]{Lemma}
\newtheorem{construction}[defn]{Construction}
\newtheorem{prop}[thm]{Proposition}
\newtheorem*{prop*}{Proposition}
\newtheorem{cor}[thm]{Corollary}
\newtheorem{conj}[thm]{Conjecture}

\newtheorem{ques}[thm]{Question}
\newtheorem{claim}[thm]{Claim}
\newtheorem*{claim*}{Claim}
\newtheorem{algo}[defn]{Algorithm}
\theoremstyle{remark}
\newtheorem{rem}[defn]{Remark}
\theoremstyle{remark}
\newtheorem{remarks}[defn]{Remarks}
\theoremstyle{remark}
\newtheorem{notation}[defn]{Notation}
\theoremstyle{remark}
\newtheorem{exmp}[defn]{Example}
\theoremstyle{remark}
\newtheorem{examples}[defn]{Examples}
\theoremstyle{remark}
\newtheorem{dgram}[defn]{Diagram}
\theoremstyle{remark}
\newtheorem{fact}[defn]{Fact}
\theoremstyle{remark}
\newtheorem{illust}[defn]{Illustration}
\theoremstyle{remark}
\newtheorem{que}[defn]{Question}
\numberwithin{equation}{section}
\newtheorem{example}[defn]{Example}
\newtheorem{exercise}[defn]{Exercise}

\def\abs#1{{\lvert#1\rvert}}
\def\MSB{MSB }
\def\pivot{\mathsf{pivot}}
\def\tred{\textcolor{red}}
\def\dim{\mathsf{dim}}
\def\MS{\mathsf{MS}}
\def\res{\upharpoonright}
\def\tensor{\bigotimes}
\def\defeq{\vcentcolon=}
\def\meet{\wedge}
\def\join{\vee}
\def\dLOfpb#1#2{\mathrm{dLO}_\mathrm{fp}^{{#1}{#2}}}
\def\corner{F}
\def\less{\prec}
\def\frame#1{\begin{mdframed}#1\end{mdframed}}
\def\hdammock{$\Tilde{h}$ammock }
\def\bb{\mathfrak{b}}
\renewcommand{\labelitemii}{$ \circ $}
\def\sd{sd}
\def\eqdef{=\vcentcolon}
\def\upset{\uparrow}
\def\downset{\downarrow}
\def\length{\mathsf{length}}
\def\gray{\textcolor{gray}}
\def\teal{\textcolor{teal}}
\def\lime{\textcolor{lime}}
\def\magenta{\textcolor{magenta}}
\def\orange{\textcolor{orange}}
\newcommand{\Il}{\pi_l(\ii)}
\def\tblue{\textcolor{blue}}
\newcommand{\Ir}{\pi_r(\ii)}
\newcommand{\htIl}{\htpi_l(\ii)}
\newcommand{\htIr}{\htpi_r(\ii)}
\newcommand{\STT}{\mathsf{long}}
\newcommand{\smol}{\mathsf{short}}
\newcommand{\beeg}{\mathsf{long}}
\newcommand{\eqvl}{\equiv_H^l}
\newcommand{\eqvr}{\equiv_H^r}
\newcommand{\eqvlr}{\equiv_H^{lr}}
\newcommand{\eqv}{\equiv_{H}}

\newcommand{\Lo}{\mathsf{L^0}}
\newcommand{\Lp}{\mathsf{L^+}}
\newcommand{\Lm}{\mathsf{L^-}}
\newcommand{\sfL}{\mathsf{L}}
\newcommand{\Ro}{\mathsf{R^0}}
\newcommand{\Rp}{\mathsf{R^+}}
\newcommand{\Rm}{\mathsf{R^-}}
\newcommand{\sfR}{\mathsf{R}}

\newcommand{\DL}{\Delta_\sfL}
\newcommand{\DR}{\Delta_\sfR}

\newcommand{\hb}{\widebar{H}}
\newcommand{\hht}{\widehat{H}}
\newcommand{\hd}{\widetilde H}
\newcommand{\Hb}{\overline{\mathcal H}}
\newcommand{\Hht}{\widehat{\mathcal H}}
\newcommand{\Hd}{\widetilde{\mathcal H}}
\newcommand{\infb}{\,^\infty\bb}
\newcommand{\binf}{\,\bb^\infty}
\newcommand{\qb}{\mathsf Q^\mathsf {Ba}}
\newcommand{\suc}{\mathfrak{succ}}
\newcommand{\pred}{\mathfrak{pred}}
\newcommand\A{\mathcal{A}}
\newcommand\B{\mathsf{B}}
\newcommand\BB{\mathcal{B}}
\newcommand\C{\mathcal{C}}
\newcommand\Pp{\mathcal{P}}
\newcommand\D{\mathcal{D}}
\newcommand\Hamm{\hat{H}}
\newcommand\hh{\mathfrak{h}}
\newcommand\HH{\mathcal{H}}
\newcommand\RR{\mathcal{R}}
\newcommand\Red[1]{\mathrm{R}_{#1}}
\newcommand\HRed[1]{\mathrm{HR}_{#1}}
\newcommand\K{\mathcal{K}}
\newcommand\LL{\mathcal{L}}
\newcommand\Lim{\text{{\normalfont Lim}}}
\newcommand\M{\mathcal{M}}
\newcommand\SD{\mathcal{SD}}
\newcommand\MD{\mathcal{MD}}
\newcommand\SMD{\mathcal{SMD}}
\newcommand\T{\mathcal{T}}
\newcommand\TT{\mathfrak T}
\newcommand\ii{\mathcal I}
\newcommand\UU{\mathcal{U}}
\newcommand\VV{\mathcal{V}}
\newcommand\ZZ{\mathcal{Z}}

\newcommand\Q{\mathbb{Q}}
\newcommand{\N}{\mathbb{N}} 
\newcommand{\R}{\mathbb{R}}
\newcommand{\Z}{\mathbb{Z}}
\newcommand{\qq}{\mathfrak q}
\newcommand{\ch}{\circ_H}
\newcommand{\cg}{\circ_G}
\newcommand{\bua}[1]{\mathfrak b^{\alpha}(#1)}
\newcommand{\falpha}{{\mathfrak{f}\alpha}}
\newcommand{\fgamma}{\gamma^{\mathfrak f}}
\newcommand{\fbeta}{{\mathfrak{f}\beta}}
\newcommand{\bub}[1]{\mathfrak b^{\beta}(#1)}
\newcommand{\bla}[1]{\mathfrak b_{\alpha}(#1)}
\newcommand{\blb}[1]{\mathfrak b_{\beta}(#1)}
\newcommand{\lmin}{\lambda^{\mathrm{min}}}
\newcommand{\lmax}{\lambda^{\mathrm{max}}}
\newcommand{\xmin}{\xi^{\mathrm{min}}}
\newcommand{\xmax}{\xi^{\mathrm{max}}}
\newcommand{\lbmin}{\bar\lambda^{\mathrm{min}}}
\newcommand{\lbmax}{\bar\lambda^{\mathrm{max}}}
\newcommand{\ff}{\mathfrak f}
\newcommand{\cc}{\mathfrak c}
\newcommand{\dd}{\mathfrak d}
\newcommand{\sqsf}{\sqsubset^\ff}
\newcommand{\rr}{\mathfrak r}
\newcommand{\pp}{\mathfrak p}
\newcommand{\uu}{\mathfrak u}
\newcommand{\vv}{\mathfrak v}
\newcommand{\ww}{\mathfrak w}
\newcommand{\xx}{\mathfrak x}
\newcommand{\yy}{\mathfrak y}
\newcommand{\zz}{\mathfrak z}
\newcommand{\MM}{\mathfrak M}
\newcommand{\mm}{\mathfrak m}
\newcommand{\sbq}{\mathfrak s}
\newcommand{\tbq}{\mathfrak t}
\newcommand{\Spec}{\mathbf{Spec}}
\newcommand{\Br}{\mathbf{Br}}
\newcommand{\sk}[1]{\{#1\}}
\newcommand{\Prime}{\mathbf{Pr}}
\newcommand{\Parent}{\mathbf{Parent}}
\newcommand{\Uncle}{\mathbf{Uncle}}
\newcommand{\Cousin}{\mathbf{Cousin}}
\def\sgn{\mathrm{sgn}}

\newcommand{\Nephew}{\mathbf{Nephew}}
\newcommand{\Sibling}{\mathbf{Sibling}}
\newcommand{\uc}{\mathrm{uc}}
\newcommand{\MCP}{\mathrm{MCP}}
\newcommand{\MSCP}{\mathrm{MSCP}}
\newcommand{\TTT}{\widetilde{\T}}
\newcommand{\la}{l}
\newcommand{\ra}{r}
\newcommand{\lb}{\bar{l}}
\newcommand{\rb}{\bar{r}}
\newcommand{\tBa}{\varepsilon^{\mathrm{Ba}}}
\def\ii{\mathcal{I}}
\newcommand{\brac}[2]{\langle #1,#2\rangle}
\newcommand{\braket}[3]{\langle #1\mid #2:#3\rangle}
\newcommand{\fin}{fin}
\newcommand{\inff}{inf}
\newcommand{\Zg}{\mathrm{Zg}(\Lambda)}
\newcommand{\Zgs}{\mathrm{Zg_{str}}(\Lambda)}
\newcommand{\STR}[1]{\mathrm{Str}(#1)}
\newcommand{\dmod}{\mbox{-}\operatorname{mod}}
\newcommand{\HQ}{\mathcal{HQ}^\mathrm{Ba}}
\newcommand{\bHQ}{\overline{\mathcal{HQ}}^\mathrm{Ba}}
\newcommand\Af{\mathcal{A}^{\ff}}
\newcommand\AAf{\bar{\mathcal{A}}^{\ff}}
\newcommand\Hf{\mathcal{H}^{\ff}}
\newcommand\Rf{\mathcal{R}^{\ff}}
\newcommand\Tf{\T^{\ff}}
\newcommand\Uf{\mathcal{U}^{\ff}}
\newcommand\Sf{\mathcal{S}^{\ff}}
\newcommand\Vf{\mathcal{V}^{\ff}}
\newcommand\Zf{\mathcal{Z}^{\ff}}
\newcommand\bVf{\overline{\mathcal{V}}^{\ff}}
\newcommand\bTf{\overline{\mathcal{T}}^{\ff}}
\newcommand{\fmin}{\xi^{\mathrm{fmin}}}
\newcommand{\fmax}{\xi^{\mathrm{fmax}}}
\newcommand{\xif}{\xi^\ff}
\newcommand{\mycomment}[1]{}
\newcommand{\LOfp}{\mathrm{LO}_{\mathrm{fp}}}

\newcommand\Tl{\mathbf{T}}
\newcommand\Tla{\mathbf{T}_{\la}}
\newcommand\Tlb{\mathbf{T}_{\lb}}
\newcommand\Ml{\mathbf{M}}
\newcommand\Mla{\mathbf{M}_{\la}}
\newcommand\Mlb{\mathbf{M}_{\lb}}
\newcommand\OT{\mathcal{O}}
\newcommand\LO{\mathbf{LO}}

\def\tblue{\textcolor{blue}}

\newcommand\rad{\mathrm{rad}_\Lambda} 

\newcommand{\fork}[1]{\mathrm{Str}_{\text{Fork}}^{\la}(#1)}
\newcommand{\BaB}{\mathsf{Ba(B)}}
\newcommand{\CycB}{\mathsf{Cyc(B)}}
\newcommand{\ExtB}{\mathsf{Ext(B)}}
\newcommand{\St}{\mathsf{St}}
\def\short{\mathsf{short}}
\newcommand{\StB}[1]{\mathsf{St}_{#1}(\sB)}
\newcommand{\STB}[1]{\mathsf{St}_{#1}(\xx_0,i;\sB)}
\newcommand{\BST}[1]{\mathsf{BSt}_{#1}(\xx_0,i;\sB)}
\newcommand{\CSt}[1]{\mathsf{CSt}_{#1}(\sB)}
\newcommand{\CST}[1]{\mathsf{CSt}_{#1}(\xx_0,i;\sB)}
\newcommand{\ASt}[1]{\mathsf{ASt}_{#1}(\sB)}

\newcommand{\OSt}[1]{\overline{\mathsf{St}}_{#1}(\sB)}
\newcommand{\OST}[1]{\overline{\mathsf{St}}_{#1}(\xx_0,i;\sB)}
\newcommand{\lB}{\ell_{\sB}}
\newcommand{\lbB}{\overline{\ell}_{\sB}}
\newcommand{\LB}{l_{\sB}}
\newcommand{\LbB}{\overline{l}_{\sB}}

\newcommand{\BalB}{\mathsf{Ba}_l(\sB)}
\newcommand{\BalbB}{\mathsf{Ba}_{\lb}(\sB)}
\newcommand{\BarB}{\mathsf{Ba}_r(\sB)}
\newcommand{\BarbB}{\mathsf{Ba}_{\bar{r}}(\sB)}

\newcommand{\sB}{\mathsf{B}}
\newcommand{\Str}[1]{\mathrm{Str}_{#1}(\xx_0,i;\sB)}
\newcommand{\Strd}[1]{\mathrm{Str}'_{#1}(\xx_0,i;\sB)}
\newcommand{\Strdd}[1]{\mathrm{Str}''_{#1}(\xx_0,i;\sB)}
\newcommand{\Cent}{\mathrm{Cent}(\xx_0,i;\sB)}
\newcommand{\Start}{\mathrm{Start}(\xx_0,i;\sB)}
\newcommand{\End}{\mathrm{End}(\xx_0,i;\sB)}
\newcommand{\braclB}{\brac{1}{\lB}}
\newcommand{\braclbB}{\brac{1}{\lbB}}
\newcommand{\QBa}{\Q^{\mathrm{Ba}}}
\newcommand{\Ba}{\mathcal Q_0^\mathrm{Ba}}
\newcommand{\Cyc}{\mathsf{Cyc}(\Lambda)}
\newcommand{\lazy}{1_{(v, i)}}
\newcommand{\dLOfd}{\mathrm{dLO_{fd}}}
\newcommand{\dLOfdb}[2]{\mathrm{dLO}_{\mathrm{fd}}^{{#1}{#2}}}

\newcommand{\LOfd}{\mathrm{LO}_{\mathrm{fd}}}
\newcommand{\Balr}{((\BalB\cap \BarB) \cup (\BalbB \cap \BarbB))}
\newcommand{\Hbar}{\overline{H}_l^i(\xx_0)}
\newcommand{\Hhat}{\widehat{H}_l^i(\xx_0)}
\newcommand{\HOST}[1]{\widehat{\overline{\mathsf{St}}}_{#1}(\xx_0,i;\sB)}
\newcommand{\HHlix}{\widehat{H}_l^i(\xx_0)}

\def\lan{\langle}
\def\ran{\rangle}
\newcommand{\s}[1]{{\mathsf{#1}}}
\newcommand{\f}[1]{\mathfrak {#1}}
\newcommand{\ol}[1]{\overline{#1}}

\newcommand{\gst}{\s{St}}
\newcommand{\stbar}{\ol\gst}
\mathchardef\mh="2D 
\newcommand{\allst}{(\le\omega)\mh\s{St}(\s B)}
\def\da{\Downarrow}
\def\ua{\Uparrow}
\def\dotplus{\cdot+}
\def\rad{\text{rad}}
\def\comp{\circ}
\def\aa{\mathfrak a}
\def\to{\rightarrow}
\def\HR{\textsf{HR}}
\def\parens#1{{(#1)}}
\def\mbf{\mathbf}
\newcommand{\rank}{\mathsf{rank}}
\newcommand{\mfa}{\mathfrak{a}}
\def\lin{\mathbf{L}}
\def\two{\mathbbm{2}}

\makeatletter
\let\save@mathaccent\mathaccent
\newcommand*\if@single[3]{%
  \setbox0\hbox{${\mathaccent"0362{#1}}^H$}%
  \setbox2\hbox{${\mathaccent"0362{\kern0pt#1}}^H$}%
  \ifdim\ht0=\ht2 #3\else #2\fi
  }
\newcommand*\rel@kern[1]{\kern#1\dimexpr\macc@kerna}
\newcommand*\widebar[1]{\@ifnextchar^{{\wide@bar{#1}{0}}}{\wide@bar{#1}{1}}}
\newcommand*\wide@bar[2]{\if@single{#1}{\wide@bar@{#1}{#2}{1}}{\wide@bar@{#1}{#2}{2}}}
\newcommand*\wide@bar@[3]{%
  \begingroup
  \def\mathaccent##1##2{%
    \let\mathaccent\save@mathaccent
    \if#32 \let\macc@nucleus\first@char \fi
    \setbox\z@\hbox{$\macc@style{\macc@nucleus}_{}$}%
    \setbox\tw@\hbox{$\macc@style{\macc@nucleus}{}_{}$}%
    \dimen@\wd\tw@
    \advance\dimen@-\wd\z@
    \divide\dimen@ 3
    \@tempdima\wd\tw@
    \advance\@tempdima-\scriptspace
    \divide\@tempdima 10
    \advance\dimen@-\@tempdima
    \ifdim\dimen@>\z@ \dimen@0pt\fi
    \rel@kern{0.6}\kern-\dimen@
    \if#31
      \overline{\rel@kern{-0.6}\kern\dimen@\macc@nucleus\rel@kern{0.4}\kern\dimen@}%
      \advance\dimen@0.4\dimexpr\macc@kerna
      \let\final@kern#2%
      \ifdim\dimen@<\z@ \let\final@kern1\fi
      \if\final@kern1 \kern-\dimen@\fi
    \else
      \overline{\rel@kern{-0.6}\kern\dimen@#1}%
    \fi
  }%
  \macc@depth\@ne
  \let\math@bgroup\@empty \let\math@egroup\macc@set@skewchar
  \mathsurround\z@ \frozen@everymath{\mathgroup\macc@group\relax}%
  \macc@set@skewchar\relax
  \let\mathaccentV\macc@nested@a
  \if#31
    \macc@nested@a\relax111{#1}%
  \else
    \def\gobble@till@marker##1\endmarker{}%
    \futurelet\first@char\gobble@till@marker#1\endmarker
    \ifcat\noexpand\first@char A\else
      \def\first@char{}%
    \fi
    \macc@nested@a\relax111{\first@char}%
  \fi
  \endgroup
}
\makeatother
\newcommand\test[1]{%
$#1{M}$ $#1{A}$ $#1{g}$ $#1{\beta}$ $#1{\mathcal A}^q$
$#1{AB}^\sigma$ $#1{H}^C$ $#1{\sin z}$ $#1{W}_n$}

\def\eps{\varepsilon}

\def\beps{\bm{\eps}}
\def\inorder{<_\two}
\def\emptyseq{\lan\ran}
\def\LOfa{\mathrm{LO_{fa}}}

\def\R{\mathbb{R}}
\def\diam{\diamondsuit}
\def\club{\clubsuit}
\def\P{\mathcal{P}}
\def\res{\upharpoonright}
\def\name{\mathring}
\def\A{\mathcal{A}}
\def\dom{\text{dom}}
\def\K{\mathcal{K}}
\def\Pbar{\overline{P}}
\def\bf{\mathbf}
\def\lessp{\leq^{pr}_}
\def\lessap{\leq^{apr}_}
\def\lessstar{\leq_{\star}}
\def\I{\bf{I}}

\def\vsk{\vspace{6pt}}
\def\cal{\mathcal}
\def\frak{\mathfrak}
\def\e{\varepsilon}
\def\club{\clubsuit}
\def\diam{\diamondsuit}
\def\Cohen{\textsf{Cohen}}
\def\Random{\textsf{Random}}
\def\cf{\textsf{cf}}
\def\dom{\mathsf{dom}}
\def\rng{\mathsf{range}}
\def\supp{\textsf{supp}}
\def\odim{\textsf{odim}}
\def\cont{\frak{c}}
\def\lt{\textsf{left}}
\def\rt{\textsf{right}}
\def\up{\textsf{up}}
\def\down{\textsf{down}}
\def\otp{\textsf{otp}}
\def\rk{\textrm{rk}}
\def\opt{\textsf{otp}}
\def\add{\textsf{add}}
\def\split{\textsf{Split}}
\def\Meager{\textsf{Meager}}
\def\Null{\textsf{Null}}
\mathchardef\mhyphen="2D
\def\blank{\mhyphen}
\def\deltahat{\hat{\delta}}
\def\call{\mathcal}
\def\red{\textcolor{red}}
\def\blue{\textcolor{blue}}
\def\LOfa{\mathrm{LO}_\mathrm{fa}}
\def\M{{M_{1b}}}
\def\m{{M_{1a}}}
\def\preorder{\text{$<_\textit{pre} $}}
\def\front{\mathrm{front}}
\def\Front{\mathbf{Front}}
\def\plus{{+}}
\def\minus{{-}}
\def\layer{\textsf{layer}}

\begin{abstract}
   The shuffle of a non-empty countable  set $ S $ of linear orders is the (unique up to isomorphism)  linear order $ \Xi(S) $ obtained by fixing a coloring function $ \chi: \mathbb{Q} \to S $ having  fibers dense in $ \mathbb{Q} $ and replacing each rational $ q $ in $ (\mathbb{Q}, <) $ with an isomorphic copy of $ \chi(q) $. We prove that any two countable shuffles that embed as convex subsets into each other are order isomorphic. 
\end{abstract}
\maketitle

\vspace{-6mm}
\renewcommand{\phi}{\varphi}
\def\Lin{\mathbf{Lin}}
\def\C{\mathbf{C}}
\def\Obj{\mathrm{Obj}}
\def\Sets{\mathbf{Sets}}

\section{Introduction}
\label{S: Introduction}
\def\ZFC{\textsf{ZFC}}
\def\ZF{\textsf{ZF}}
\def\AC{\textsf{AC}}
The set-theoretic Cantor-Schr\"oder-Bernstein theorem (henceforth the CSB theorem) guarantees a bijection between sets $ A $ and $ B $  whenever there exist injections $ f_A : A \to B $ and $ f_B : B \to A $. Analogues of this theorem are known to hold in many settings, including those of $ \sigma $-complete Boolean algebras \cite{Sikorski1948, tarskicardinal},  
 computable sets (the Myhill isomorphism theorem \cite[Theorem~18]{myhill1955creative}), 
 measurable spaces \cite[\S~3.3, Proposition~3.3.6]{srivastava2008course}, 
 and 
 $ \infty $-groupoids \cite{escardo} 
 under appropriate maps. Often such results require an appropriate strengthening of the notion of an injection. 

It is well known that the CSB theorem does not hold for linear orders with monotone maps. This is true even when we consider, more specifically, just the embeddings with convex images (Definition~\ref{defn: convexity}): there exist  pairs of linear orders, such as $ (\Q, <) $ and $ ([0, 1] \cap \Q, <) $, that are \emph{convex-biembeddable} but not isomorphic. Therefore, we restrict our attention to the class of countable \emph{shuffles}.

The shuffle operator $ \Xi $, first introduced by Skolem in \cite{Skolem1970-SKOLUB}, produces, for any given countable set $ S $ of linear orders (the \emph{shufflands}), a single linear order $ \Xi(S) $ by replacing each element of $ \Q $ by one of the shufflands in a homogeneous manner (Definition~\ref{defn: shuffle}). 

By closing the set of suborders of $ (\N, <) $ under order isomorphisms,  and the basic operations of order reversal, binary addition, lexicographic multiplication and finitary shuffles, Lauchli and Leonard  \cite{leonard1968elementary} obtained a class $ \mathcal{M} $ of linear orders having computable $ n $-types. In the same paper, they used this class to prove the decidability of the first order theory of (colored) linear orders through Ehrenfeucht-Fr\"aiss\'e games. The first author encountered the orders in $ \mathcal{M} $ under the name of \emph{finite description linear orders} while computing an ordinal-valued Morita-invariant for certain finite-dimensional associative algebras known as special biserial algebras \cite{SSK}.  
Together with Kuber \cite{SK}, he obtained an alternate proof using automata of the fact that certain linear orders key to this computation lie in the class $ \mathcal{M} $.
The second author's interest in linear orders stems from a previous work \cite{mittal2024exponentiablelinearordersneed} on another operation on the class of linear orders, namely exponentiation. 

 
The following theorem is the main result of this paper.

\begin{thm}
\label{thm: main theorem}
    Let $ L_\plus, L_\minus $ be convex-biembeddable countable linear orders and suppose that there exist countable sets $ S^0_\plus, S^0_\minus $ of linear orders such that $ L_\plus \cong \Xi(S^0_\plus) $ and $ L_\minus \cong \Xi(S^0_\minus) $. Then $ L_\plus \cong L_\minus $.
\end{thm}

For the remainder of the paper, we fix $ L_i, S^0_i $ for $ i \in \{\plus, \minus\} $ as in the hypothesis of Theorem~\ref{thm: main theorem}. 

We take inspiration from
Dedekind's proof of the set-theoretic CSB theorem within $ \ZF $ (Zermelo–Fraenkel set theory without the Axiom of Choice), which essentially involved constructing $ \lan (A_n, B_n) \mid n \in \omega \ran $ with $ A_0 := A, B_0 := B, A_{n + 1} := f_B(B_n) $ and $ B_{n + 1} := f_A(A_n) $. Observing that $ A_n \supseteq A_{n + 1} $ for each $ n $ and letting $ A_{\omega} := \bigcap_{n \in \omega} A_n $ yields a useful partition $ A = \bigsqcup_{n \in \omega} (A_n \setminus A_{n + 1}) \sqcup A_{\omega} $ (and likewise for $ B $). Similarly, we use the given convex embeddings $f_i$ from $L_i$ to  $L_{-i}$ to construct multifunctions $g_i$ from $L_i$ to $L_{-i}$. We stratify $ L_i $ into $ \omega + 1 $ descending levels $ \lan A_{i, n} \mid n \in \omega + 1 \ran $ for each $ i \in \{\plus, \minus\} $ using $A_{i,n+1} = g_{-i}(A_{-i,n})$ and $ A_{i, \omega} = \bigcap_{n \in \omega} A_{i, \omega} $.





We then construct, for each $ i \in \{\plus, \minus\} $ a colored tree $ T'_i $ whose leaf nodes index a partition of $ L_i \setminus A_{i, \omega} $ into intervals
and whose infinite branches index a partition of $ A_{i, \omega} $ into intervals. 
Studying the colored linear orders $ \Front(T_i') $, $ \ol\Front(T_i') $ formed by the leaf nodes and branches of $ T_i' $ (Definition~\ref{defn: preliminary tree}) yields expressions for the order isomorphism types of $ L_i \setminus A_{i, \omega} $ and $ L_i $ respectively that do not depend on $ i $, thereby completing our proof. 

The rest of the paper is organized as follows. We recall standard results and establish our notation for linear orders in Section~\ref{S: Lin Preliminaries} and for trees in Section~\ref{S: Tree Preliminaries}. In Section~\ref{S: Constructing}, we construct a family of embeddings that we then use to partition $ L_i $. We then construct colored trees $ T_i $ and $ T_i' $ mimicking these embeddings, and use their nodes and branches  to refine our partition.
In Section~\ref{S: Proving}, we prove isomorphisms between the frontiers of $ T_i, T_i' $ and (suborders of) $ L_i, L_{-i} $ to complete the proof.

{We use the notation $ \sqcup $ to denote disjoint unions {and power set by $\mathcal{P}$.}

\subsection*{Acknowledgements}
The authors are grateful to Dr.~Amit Kuber for his careful reading of the manuscript and his valuable suggestions to improve the writing.



\section{Preliminaries on linear orders}
\label{S: Lin Preliminaries}

\begin{defn}

    \label{defn: convexity}
A \emph{linear order} $ (L, <) $, or simply $ L $ for short, is a set $ L $ endowed with a binary relation $ < $ that is {irreflexive, transitive and total (any two distinct elements are comparable)}. 
Morphisms between linear orders are monotone functions, and bijective morphisms are isomorphisms.

 Given $ A, B \subseteq L $, we say that $ A \ll B $ if $ a < b $ whenever $ a \in A, b \in B $. We say that $ A \subseteq L $ is \emph{convex} if whenever $ a \in L \setminus A $, then either $ \{a\} \ll A $ or $ A \ll \{a\} $.

 A monotone map between linear orders is said to be a \emph{convex embedding} if its image is convex.
Two linear orders are said to be \emph{convex-biembeddable} if there is a convex embedding from either one to another.

\end{defn}


\begin{example}
    For each $ n \in \N $, $ (\{0, 1, \dots, n - 1\}, <) $, where $ < $ is the usual ordering of the reals, is a linear order denoted $ \mbf n $. Similarly, $ (\Q, <) $ is the linear order of the rationals and $ (\ol{\Q}, <) $ that of the \emph{closed rationals} $ \{-\infty\} \sqcup \Q \sqcup \{\infty\} $, with $ -\infty, \infty $ being the least and the greatest elements in $(\ol{\Q},<)$ respectively.
\end{example}


The following theorem is due to Cantor. 
\begin{thm}\label{thm:Cantor}
    Every countable linear order is embeddable into rationals.
\end{thm}

\begin{defn}
\label{defn: coloring}
Given a linear order $ L $ and a non-empty set $ S $, an \emph{$ S $-coloring} of $ L $ is a function $ \chi : L \to S $.

    An $ S $-coloring $ \chi $ of $ L $ is said to be \emph{dense}, if for every $ s \in S $ and $ x < x'$ in $ L $, there is some $ y \in L $ with $ x < y < x' $ and $ \chi(y) = s $.
\end{defn}

\begin{rem}
    Given a linear order $L$ and a set $S$, a dense $S$-coloring of $L$ is surjective, and hence $ S $ is countable.
\end{rem}

The archetypal operation on linear orders is that of taking \emph{ordered sums}.

\begin{defn}\label{defn: 5}
\label{defn: ordered sum}
    Given a linear order $ L $, a non-empty set $ S $ of linear orders, and an $ S $-coloring $ \chi : L \to S $, the \emph{ordered sum} $ \sum_{x \in L} \chi(x) $ is defined to be the set $ \bigsqcup_{x \in L} (\{x\} \times \chi(x)) $ under the relation $ (x, y) < (x', y') $ if  either $ x < x' $ or $ x = x' $ and $ y < y' $.
\end{defn}

The existence and essential uniqueness of dense colorings of $ \Q $ by a given countable set $ S $ is the substance of the following famous theorem due to Skolem.

\begin{thm}[\cite{rosenstein}, Theorems~7.11 and 7.13]
\label{thm: cantor}
    Let $ S $ be a non-empty countable set. Then there exists a dense coloring $ \chi: \Q \to S $. Furthermore, this coloring is unique: for any dense coloring $ \chi' : \Q \to S $, there exists an order isomorphism $ \iota : \Q \to \Q $ with $ \chi' = \chi \circ \iota $.
\end{thm}

This theorem enables one to define the \emph{shuffles} of countable sets of linear orders.

\begin{defn}
\label{defn: shuffle}
A linear order $ L $ is said to be a \emph{shuffle linear order} (or \emph{shuffle}, for short), if there exists a set of linear orders $ S $ along with a dense coloring $ \chi : \Q \to S $ such that $ L \cong \sum_{r \in \Q} \chi(r) =: \Xi(S) $.
\end{defn}

\begin{rem}\label{rem: zero not in S}
    For any countable set $S$ of linear orders $\Xi(S) \cong \Xi(S \cup \{ \mbf 0 \})$.
\end{rem}

The following proposition states some properties of the shuffles. 
\begin{prop}[\cite{heilbrunner1980algorithm}]
\label{prop: property_shuffle}
    Let $S$ be a countable set of linear orders. Then, $\Xi(S)$ satisfy the following properties:
    \begin{enumerate}
        \item $\Xi(S) + \Xi(S) \cong \Xi(S)$;
        \item $\Xi(S) + s + \Xi(S) \cong \Xi(S)$ for every $s \in S$; and
        \item $\Xi(S' \cup S'') \cong \Xi(S)$, where $S' \subseteq S$ and $\emptyset \neq S'' \subseteq \{ t_1 + \Xi(S) + t_2 \mid t_1,t_2 \in S \cup \{ \mbf 0 \} \}$. 
    \end{enumerate} 
    
    In particular, the third property yields that $\Xi(S \sqcup \{ s \}) \cong \Xi(S)$, where $s :=\Xi(S)$.
\end{prop}

\section{Preliminaries on trees}
\label{S: Tree Preliminaries}
\begin{defn}
    For a set $ X $ and $ n \in \N $, by an \emph{$ n $-sequence} of elements in $ X $ we mean a function $ f: \{0, 1, \dots,  n - 1\} \to X $, also denoted $ \lan f(0), f(1), \dots, f(n-1) \ran $. In the case $ n = 0 $, we get the \emph{empty sequence}, denoted $ \lan \ran $. The set $ \bigcup_{n \in \omega} X^n $ of all \emph{finite sequences} of elements in $ X $ is denoted by $ X^{< \omega} $.
    
    Likewise, the set of all \emph{$ \omega $-sequences} of elements in $ X $, i.e., functions from $ \omega $ to $ X $, is denoted $ X^\omega $, while the union $ X^{< \omega} \sqcup X^\omega $ is denoted $ X^{\le \omega} $.  \emph{Concatenation} of pairs of sequences in $ X^\omega \times X^{\le \omega} $ is an associative noncommutative operation denoted by juxtaposition.




\end{defn}

\begin{example}\label{example: seq}
    The sequences $\sigma = \lan  2, 1, \infty \ran$ and $\tau = \lan 2, 3, -\infty \ran$ concatenate to form $\sigma \tau = \lan 2, 1, \infty, 2, 3, -\infty \ran$ with $\abs{\sigma \tau} = \abs{\sigma} + \abs{\tau} = 2 + 2 = 4 $. 
\end{example}

There are two natural partial order relations on $ \ol\Q^{\le\omega}$.

\begin{defn}
    Let $ \sqsubseteq, \prec $ be partial order relations on $ \ol\Q^{\le \omega} $ such that for any $ \sigma, \tau \in \ol\Q^{\le \omega} $ \begin{itemize}
        \item  $ \sigma \sqsubseteq \tau $ if and only if $ \sigma $ is a prefix of $ \tau $, i.e., $ \tau = \sigma\rho $ for some $ \rho \in \ol\Q^{\le \omega} $; and

        \item $ \sigma \prec \tau $ if and only if for some $ \rho \in \ol\Q^{< \omega} $ and $ q, q' \in \ol\Q $, we have $ \rho \lan q\ran  \sqsubseteq \sigma $, $ \rho \lan q' \ran \sqsubseteq \tau $ and $ q < q' $.
    \end{itemize}
\end{defn}

Note that $ \sigma $ and $ \tau $ are $ \prec $-comparable if and only if $ \sigma \not\sqsubseteq \tau \text{ and } \tau \not\sqsubseteq \sigma $. 

{For any sequence $ x \in \ol{\Q}^{< \omega} $, define $ \abs{x} $ to be the number of occurrences of rationals in $ x $.}

\begin{example}
 In Example~\ref{example: seq}, $\sigma \sqsubseteq \sigma \tau$ but $\tau \not \sqsubseteq \sigma \tau$ because $\sigma \tau \neq \tau \rho$ for any $\rho \in \ol\Q^{\le \omega}$. Also, $\sigma \prec \tau$ since they both differ first at index $1$, where $\sigma(1) = 1 < 3 = \tau(1)$.
\end{example} 

\begin{defn}
\label{defn: preliminary tree}
    A \emph{tree} is a downwards-closed subposet $ T $ of the poset $ ({\ol{\Q}}^{< \omega}, \sqsubseteq) $. 
    
    Elements of a tree are called its \emph{nodes}. If $ \tau, \sigma \in T $ with $ \sigma = \tau \lan q \ran $ for some $ q \in \ol\Q $, then $ \tau $ is said to be the \emph{parent} of its \emph{child }$ \sigma $. Nodes with no children are called \emph{leaf} nodes. 
    An infinite sequence $ r \in \ol\Q^\omega $ is said to be a  \emph{branch} of $ T $ if $ r \res n \in T $ for every $ n \in \N $. 

    The set of leaf nodes of a tree $ T $ is called its \emph{frontier} and denoted $ \front(T) $, while the set $ \ol \front(T) := \front(T) \cup \{r \in \ol{\Q}^\omega \mid r $ is a branch of $ T \} $ of leaf nodes and branches is called the \emph{closed frontier} of $ T $.  For any $ n \in \N $, the \emph{$ n^{th} $-layer} $ \layer(T, n) $ of $ T $ is the set $ \{x \in \front(T) \mid \abs{x} = n + 2\} $.

    Under $ \prec $, the frontier $ \front(T) $ of any tree $ T $ forms a linear order $ \Front(T) $ called its \emph{frontier}. Similarly, $ \ol{\Front}(T) := (\front(T), \prec) $ is also a linear order.

    For any set $ S $, the pair $ (T, c:T \to S) $ is called a colored tree, where $c$ is an \emph{$ S $-coloring} of $ T $. Thus, it is easy to check that $ (\Front(T), c: \Front(T) \to S) $ is a colored linear order.
\end{defn}
	

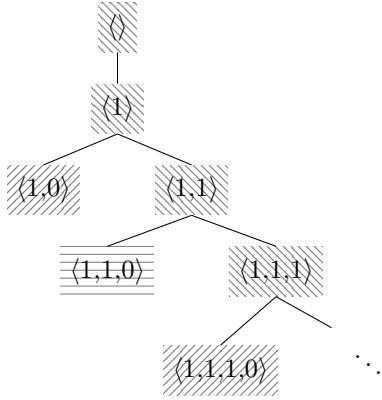
\begin{figure}[h] 
    \centering 
    \begin{forest}
    [$ \lan \ran $, s sep=10mm, for tree={pattern = north west lines, pattern color = gray}
        [$\lan 1 \ran$, s sep=10mm, 
            [$\lan 1\text{$,$} 0 \ran$, for tree={pattern = north east lines, pattern color = gray}] 
            [$\lan 1\text{$,$}1\ran$, s sep = 10 mm, 
                [$\lan 1\text{$,$}1\text{$,$}0 \ran$, for tree={pattern=horizontal lines, pattern color=gray}]
                [$\lan 1\text{$,$}1\text{$,$}1 \ran$, s sep = -1 mm, 
                    [$\lan 1\text{$,$}1\text{$,$}1\text{$,$}0 \ran$, for tree={pattern = north east lines, pattern color = gray}]
                    [\;\;\;\;\;\;\;\;\;\;$ \ddots $, for tree={fill=none, draw=none}]
                ]
            ]
        ]
    ]
    \end{forest}
    \caption{An example of a tree} 
    \label{fig:tree-exmp} 
\end{figure}

\begin{example}
    In Figure \ref{fig:tree-exmp}, we consider a tri-colored tree $T$ with its bi-colored $ \front(T) $ representing all finite binary $(0\mbox{-}1)$ sequences of length atleast $2$ ending at the first occurrence of $0$. Observe that $\sigma \prec \tau$ if and only if $\abs{\sigma} < \abs{\tau}$. Thus, $\Front(T) \cong \omega$ and $\ol{\Front}(T) \cong \omega + \mbf 1$.
\end{example}












\section{Constructing embeddings, partitions, and trees}

\label{S: Constructing}



In the sequel, we use $ i $ to denote a typical element of $ \{\plus, \minus\} $. Recall from the hypothesis of Theorem~\ref{thm: main theorem} that $L_i \cong \Xi(S^0_i) $ for a countable set $S^0_i$ of linear orders. 
{Thanks to Remark~\ref{rem: zero not in S}, we can assume $\mbf 0 \not\in S^0_\plus \cup S^0_\minus$.} 
We also assume without loss of generality that the underlying sets of the linear orders in $ S_\plus^0, S_\minus^0 $ are pairwise disjoint. Moreover, we can assume that $ S^0_\plus\cup S^0_\minus\neq\emptyset $, for otherwise $ L_\plus = L_\minus = \mbf 0 $ and the conclusion is trivial. {Choose and fix convex embeddings $ f_i : L_i \to L_{-i} $ and $ \alpha_i^0, \delta_i^0 \subseteq L_{-i} $ so that $ L_{-i} = \alpha^0_{-i} \sqcup f_i(L_i) \sqcup \delta^0_{-i} $} with $\alpha^0_{-i} \ll f_i(L_i)\ll\delta^0_{-i} $. Set $F_i:= S^0_i \sqcup \{ \alpha^0_i \sqcup \delta^0_i \}\setminus\{\mbf0\}$. {Use an enumeration for $ F_i $ where $\alpha_i^0$ precedes $\delta_i^0$ to define the ordered sum $\sum_{L \in F_i} L $, and fix an order embedding $ h_i : \sum_{L \in F_i} L \to \Q $ using Theorem~\ref{thm: cantor}.} For each $ L' \in F_i $, the composite $ L' \hookrightarrow \sum_{L \in F_i} L \xrightarrow[]{h_i} \Q $ is an order isomorphism, and we sometimes abuse notation to denote the inverse of this composite by  $ h_{i}^{-1} $. Further set $s_i := \Xi(S^0_i)$, $S_i := S^0_i \sqcup \{ s_i \}$ and $\ol{F_i} := F_i \sqcup \{ \mbf 0 \} \sqcup \{ s_i \}$.





In view of Proposition~\ref{prop: property_shuffle}, we assume $ L_i = \sum_{q \in \Q} \chi_i(q) $ for a dense coloring $ \chi_i : \Q \to S_i $. In the sequel, we use this sum as the presentation of $L_i$, while often referring to its elements as pairs $(q,x)$, in view of Definition~\ref{defn: 5}. Since $ S^0_i \not= \emptyset $ implies that $ \abs{S_i} \ge 2 $, thanks to Theorem~\ref{thm: cantor}, we may assume that the colorings $ \chi_\plus $ and $ \chi_\minus $ are chosen to satisfy $ \chi_\plus^{-1}(s_\plus) = \chi_\minus^{-1}(s_\minus) =: R $.

\begin{defn}
For each $ q \in R $, let $ f_{i, q} : L_i \to L_{-i} $ be defined as $ f_{i, q}(x) := (q, f_i(x)) $ for $ x \in L_i $. For any sequence $ \rho = \lan q_0, q_1, \dots, q_n \ran \in R^{< \omega} $, define $ f_{i, \rho} := f_{i, q_0} \circ f_{-i, q_1} \circ \dots \circ f_{(-1)^ni, q_n} : L_{(-1)^ni} \to L_
{-i} $. For technical reasons, we also take $ f_{i, \lan \ran} $ to be the identity map on $ L_{-i} $. 
\end{defn}

The maps $f_{i,q}$ are injective, and hence so are $f_{i,\rho}$.


\begin{defn}
Set $ g_i: L_i \to \mathcal{P}(L_{-i})$ as $x \mapsto \{ f_{i,q}(x) \mid q \in R \}$.
Then $ g_i(x) \cap g_i(y) = \emptyset $ whenever $ x, y \in L_i $ are distinct.
By an abuse of notation, for any $ A \subseteq L_i $, we denote $ \bigsqcup_{x \in A} g_i(x) \in \mathcal{P}(L_{-i}) $ by $ g_i(A) $.
In addition, define $ \lan A_{i, n} \mid n  \in \N \ran $ inductively as $ A_{i, 0} := L_i $ and $ A_{i, n + 1} := g_{-i}(A_{-i, n}) $ for each $ n \in \N $.  Finally, let $ A_{i, \omega} \defeq \bigcap_{n \in \N} A_{i, n} $.
\end{defn}

\begin{rem}
\label{rem: partition}
Note that $ L_i = A_{i, 0} \supseteq A_{i, n} \supseteq A_{i, n + 1} \supseteq A_{i, \omega} $ for each $ n $, so that $ L_i = \bigsqcup_{n \in \N} (A_{i, n} \setminus A_{i, n + 1}) \sqcup A_{i, \omega} $.  
\end{rem}


\begin{rem}
    The fact that $ g_i(x), g_i(y) $ are disjoint whenever $ x, y \in L_i $ are distinct, yields $$ g_i(A_{i, \omega}) = g_i(\bigcap_{n \in \N} A_{i, n}) = \bigcap_{n \in \N} g_i(A_{i, n}) = \bigcap_{n \in \N} A_{-i, n + 1} = A_{-i, \omega}. $$
\end{rem}  




\def\star{*}
 
We now define colored trees $ T_i $,  whose layers will correspond to all but one part of the partition obtained in Remark~\ref{rem: partition}.


\begin{defn} \label{defn: tree}

Let $ (T_i, c_i: T_i \to \ol{F_\plus} \cup \ol{F_\minus}) $ be the colored tree with root $ \lan \ran$ coloured $s_i $ constructed inductively as follows: for any $ (\sigma, q) \in T_i \times \ol{\Q}$, $ \sigma \lan q \ran \in T_i $ if one of the following conditions hold.

\begin{enumerate}

\item $ c_i(\sigma) \in F_\plus \cup F_\minus$ and $ q \in h_i (c_i(\sigma)) $. In this case, set $ c_i(\sigma \lan q \ran) := \mbf 0 $.

\item $ c_i(\sigma) = s_j \text{ and } q \in \Q $. In this case, set $ c_i(\sigma \lan q \ran) := \chi_j(q) $.

\item $ c_i(\sigma) = s_j $, $ \sigma \not= \lan \ran, q = - \infty $ and $\alpha^0_j \not= \mbf 0$. In this case, set $ c_i(\sigma \lan q \ran) := \alpha^0_j $.

\item $ c_i(\sigma) = s_j $, $ \sigma \not= \lan \ran $, $ q = \infty $ and $\delta^0_j \not= \mbf 0$. In this case, set $ c_i(\sigma \lan q \ran) := \delta^0_j $.




    
\end{enumerate}

Let $ T_i' := T_i \setminus \front(T_i) $. 
\end{defn}

We make several observations.

\begin{rem}\label{rem: 1}
    Each sibling of a leaf node in $ T_i $ is again a leaf node, so $ \front(T_i') $ is the set of parents of leaf nodes of $ T_i $. In fact, $ c_i(\sigma)\in F_\plus \cup F_\minus$ if and only if $ \sigma\in\front(T_i') $.
\end{rem}

\begin{rem}\label{rem: 2}
    For any $\rho\in T_i$, $c_i(\rho)\in\{s_+,s_-\}$ if and only if $\rho\in R^{<\omega}$. {As a consequence, for any $ r \in \ol{\Q}^\omega $,} $ r $ is a branch of $ T_i $ if and only if $ r $ is a branch of $ T_i' $ if and only if $ r \in R^\omega $.
\end{rem}

\begin{rem}
\label{rem: biembeddable trees}
    For any $ n \in \N $, $ \rho \in \ol\Q^{n} $, let $ T_{i, n}^\rho :=\{\rho\sigma\in T_i \mid \sigma(0) \in \Q \} $ so that the obvious isomorphism $ (T_{(-1)^ni}, \sqsubseteq) \cong (T_{i, n}^\rho, \sqsubseteq) $ provides an embedding of $ T_{(-1)^ni} $ into $ T_i $ ``below'' any $ \rho \in R^n $.
\end{rem}

In view of the previous remark, for $ \rho \in R^{n} $, let $ A_{i, n}^\rho := \rng(f_{-i, \rho}) $. Furthermore, for each $ \sigma, \tau \in R^n $, let $ F_\sigma^\tau : A_{i, n}^\sigma \to A_{i, n}^\tau $ denote the isomorphism $ f_{-i, \tau} \circ f_{-i, \sigma}^{-1} $.

\begin{rem}
\label{rem: stability}
    For any $ \sigma, \tau \in R^n, \rho \in R^m $, the isomorphism $ F_{\sigma \rho}^{\tau \rho}$ is the restriction of $ F_{\sigma}^\tau$ to $A_{i, n + m}^{\sigma \rho} $.
\end{rem}

\begin{rem}
\label{rem: decreasing}
    For any $ \rho\lan q\ran \in R^{n+1}$, we have $$A_{i, n + 1}^{\rho \lan q \ran} = \rng(f_{-i, \rho\lan q \ran}) = f_{-i, \rho}(f_{(-1)^{n+1}i, q}(L_{(-1)^{n + 1}i})) \subseteq f_{-i, \rho}(L_{(-1)^ni}) = \rng(f_{-i, \rho}) = A_{i, n}^\rho .$$
\end{rem}

 Due to Remark~\ref{rem: decreasing},  $ \lan A_{i, n}^{r \res n} \mid n \in \N \ran $ is a $ \subseteq $-decreasing sequence of subsets of $ L_i $ for any $ r \in R^{\omega} $. We set $ A_{i, \omega}^r := \bigcap_{n \in \N} A_{i, n}^{r \res n} $ and end this section by decomposing $A_{i,\omega}$ as a disjoint union.

\begin{prop}\label{prop: Aiw partitioned into intervals}
\begin{enumerate}
    \item For $ n \in \N $ and $ \rho_1\neq \rho_2$ in $R^n $, $ A_{i, n}^{\rho_1}\cap A_{i, n}^{\rho_2}=\emptyset$. Moreover, $ \bigsqcup_{\rho \in R^n} A_{i, n}^\rho = A_{i, n} $.
    \item $ A_{i, \omega} = \bigsqcup_{r \in R^\omega} A_{i, \omega}^r $.
\end{enumerate}
\end{prop}
\begin{proof}
    $(1)$ The conclusion is trivial when $ n = 0 $, since $ R^0 = \{\lan \ran\} $ and $ A_{i, 0}^{\lan \ran} = \rng(f_{-i, \lan \ran}) = L_{i} = A_{i, 0} $.

    Now suppose the first statement holds for $ n \in \N $, and observe that for distinct $ \rho_1=\lan q_1\ran\rho'_1, \rho_2=\lan q_2\ran\rho'_2 \in R^{n + 1} $, there are two cases.
    
    If $ q_1 \not = q_2 $, then $ A_{i, n+1}^{\rho_1} \cap A_{i, n+1}^{\rho_2} = f_{-i, q_1} (A_{-i, n}^{\rho_1'}) \cap f_{-i, q_2} (A_{-i, n}^{\rho_2'}) \subseteq \rng(f_{-i, q_1}) \cap \rng(f_{-i, q_2}) = \emptyset $. 
    
    On the other hand, if $ q_1 = q_2 =: q $ then $ \rho_1' \not= \rho_2' $. By the inductive hypothesis and the injectivity of $ f_{i, q} $, we have $  A_{i, n+1}^{\rho_1} \cap A_{i, n+1}^{\rho_2} = f_{-i, q}(A_{-i, n}^{\rho_1'}) \cap f_{-i, q}(A_{-i, n}^{\rho_2'}) = f_{-i, q}(A_{-i, n}^{\rho_1'} \cap A_{-i, n}^{\rho_2'}) = f_{-i, q}(\emptyset) = \emptyset $.
    
    For the final conclusion, note that $$ A_{i, n + 1} = g_{-i}(A_{-i, n}) = \bigsqcup_{q \in R} f_{-i, q} \left(\bigsqcup_{\rho' \in R^n} A_{-i, n}^{\rho'}\right) = \bigsqcup_{q \in R, \rho' \in R^n} f_{-i, q}(A_{-i, n}^{\rho'}) = \bigsqcup_{\rho \in R^{n + 1}} A_{i, n+1}^\rho .$$
    
    $(2)$ First, observe that for any $ x \in A_{i, \omega} = \bigcap_{n \in \N} A_{i, n} $ and $ n \in \N $, there exists $ \rho_n \in R^n $ with $ x \in A_{i, n}^{\rho_n} $. Using Remark~\ref{rem: decreasing} together with $(1)$ above, we obtain $ A_{i, n + 1} \cap A_{i, n}^{\rho_n} = \bigsqcup_{q \in R} A_{i, n + 1}^{\rho_n \lan q \ran} $, so that $\{ \rho_n\mid  n\in\mathbb N\}$ is a $\sqsubseteq$-increasing chain. {Let $ r \in R^{\omega} $ be the limit of this sequence} so that $ x \in \bigcap_{n \in \N} A_{i, n}^{\rho_n} = \bigcap_{n \in \N} A_{i, n}^{r \res n} = A_{i, \omega}^r $. This proves the inclusion $A_{i, \omega}\subseteq \bigcup_{r \in R^\omega} A_{i, \omega}^r$. The reverse inclusion is clear, and hence the equality. It only remains to show that the union on the right side is disjoint. For this, note that given any distinct $ r, r' \in R^\omega$, choosing the least $ n \in \N $ such that $ r \res n \not= r' \res n $, we see that $ A_{i, \omega}^r \cap A_{i, \omega}^{r'} \subseteq A_{i, n}^{r \res n} \cap A_{i, n}^{r' \res n} = \emptyset $.
\end{proof}

\section{Proof of Theorem~\ref{thm: main theorem}}
\label{S: Proving}

\subsection{\texorpdfstring{$ \Front(T_i) \cong L_i \setminus A_{i, \omega} $}{Front(T_i) ≅ L_i \ A_{i, ω}}}


    
Recall that $\front(T_i)= \bigsqcup_{n \in \mathbb{N}}\layer(T_i, n)$.
\begin{rem}\label{rem: leaves}
Definition \ref{defn: tree} ensures that any $ \sigma \in \layer(T_i, n) $ has exactly one of the following forms: $ \rho \lan q_0, q_1 \ran $ or  $ \rho \lan q_0, \pm\infty, q_1 \ran $ for some $ \rho \in T_i $, $ q_0, q_1 \in \Q $.
\end{rem}

\begin{defn}
   For \( n \in \mathbb{N} \), define \( G_{i, n} : \layer(T_i, n) \to L_i \) by $\sigma\mapsto f_{i, \rho} (q_0, {h_{(-1)^n i}^{-1}(q_1)})$ if $\sigma$ is of the form  $\rho \lan q_0, q_1 \ran$ or $\rho \lan q_0, \pm \infty, q_1 \ran$. Further define \( G_i :\front(T_i) \to L_i \) by $G_i(\sigma):=G_{i,|\sigma|-2}(\sigma)$.
\end{defn}
Using monotonicity of $h_{\pm}$ and $f_{i,\rho}$, it is easy to see that $G_{i,n}$ is a monotone function. In fact, we will show later in Proposition~\ref{prop: Front(T) is L} that $G_i$ is also monotone.

Since $ f_{i, \lan\ran} $ is the identity map, we observe that $G_{i, 0}(\sigma) = (q_0, h_{i}^{-1}(q_1))$ if $ \sigma $ is of the form $ \lan q_0, q_1 \ran $ or $ \lan q_0, \pm \infty, q_1 \ran $.
    
\begin{rem}
For any $ n \in \N $ and $ \lan q \ran \sigma \in \layer(T_i, n+1) $, \begin{equation}\label{eqn: Gin inductive}
        G_{i, n+1}(\lan q \ran\sigma) = f_{i, q}(G_{-i, n}(\sigma)). 
    \end{equation}
\end{rem}
\begin{prop}
    For each $ n \in \N$, $ \rng(G_{i, n})=(A_{i, n} \setminus A_{i, n + 1}) $.
\end{prop}

\begin{proof}
Suppose $ n = 0 $. To see that $ \rng(G_{i, 0})= A_{i, 0} \setminus A_{i, 1}$, note that $ A_{i, 0} \setminus A_{i, 1} = L_i \setminus g_{-i}(L_{-i}) = \bigsqcup_{q \in \Q}\left( \{q\} \times \chi_i(q) \right)\setminus \bigsqcup_{q \in R} f_{-i, q}(L_{-i})  = \bigsqcup_{q \in \Q} (\{q\} \times L_q) $, where $$ L_q := \begin{cases}
        \chi_i(q) & \text{if $ q \not\in R $}{;}\\
        \alpha^0_i \sqcup \delta^0_i & \text{otherwise}.
    \end{cases} $$ 
    
If $q\notin R$, then $ \{q\} \times L_q = G_{i, 0}(\{\lan q, q' \ran \in \layer(T_i, 0) \}) $, otherwise $ \{q\} \times L_q = G_{i, 0}(\{\lan q, \pm \infty, q_0 \ran \in \layer(T_i, 0)\}) $.


    Now suppose we have proved that $ G_{j, k} : \layer(T_j, k) \to (A_{j, k} \setminus A_{j, k + 1}) $ is an order-preserving bijection for each $ j \in \{\plus, \minus\} $ and $ k \le n $. To show the same for $ G_{i, n + 1} $, we use Equation~\eqref{eqn: Gin inductive}. Suppose $ \lan q \ran \sigma, \lan q' \ran \sigma' \in \layer(T_i, n) $ with $ \lan q \ran \sigma \prec \lan q' \ran \sigma' $. If $ q < q' $, then $ G_{i, n+1}(\lan q \ran\sigma) \in \rng(f_{i, q}) \ll \rng(f_{i, q'}) \ni G_{i, n+1}(\lan q' \ran\sigma') $ and we are done. Otherwise $ \sigma, \sigma' \in \layer(T_{-i}, n) $ and $ \sigma \prec \sigma' $, so by our inductive hypothesis, $ G_{-i, n}(\sigma) < G_{-i, n}(\sigma') $ and $ G_{i, n + 1}(\lan q \ran\sigma) = f_{i, q}(G_{-i, n}(\sigma)) < f_{i, q}(G_{-i, n}(\sigma')) = G_{i, n+1}(\lan q \ran \sigma') $.
\end{proof}

\begin{prop}
\label{prop: Front(T) is L}
    The corestriction $ G_i : \Front(T_i) \to L_i \setminus A_{i, \omega} $ is an order-isomorphism.
\end{prop}

\begin{proof}
        Thanks to the above proposition, Remark~\ref{rem: partition} and the injectivity of $G_{i,n}$, it is easy to see that the said corestriction of $ G_i $ is bijective. It remains only to show that $ G_i $ is order-preserving. 
        
        Let $ \sigma, \tau \in \Front(T_i) $ satisfy $ \sigma \prec \tau $. Induct on $n$, where $\sigma \in \layer(T_i,k_1)$ and $\tau \in \layer(T_i,k_2)$ for $k_1,k_2 \leq n$. If $|\sigma|=|\tau|$, then we are done by the monotonicity of $ G_{i,|\sigma|-2} $. Thus we assume that $|\sigma|\neq|\tau|$.
        
        Furthermore, by the induction hypothesis for any $ j \in \{\plus, \minus\} $ and any $ \sigma_0, \tau_0 \in \Front(T_j) $ with $ \sigma_0 \prec \tau_0, \abs{\sigma_0} < \abs{\sigma} $ and $ \abs{\tau_0} < \abs{\tau} $, we have $ G_{j}(\sigma_0) < G_{j}(\tau_0) $.

        Let $ \rho \in \Q^{< \omega}, \sigma', \tau' \in \ol{\Q}^{< \omega} $ and $ q, q' \in \ol{\Q} $ be such that $ \sigma = \rho \lan q \ran \sigma' $, $ \tau = \rho \lan q' \ran \tau' $ and $ q < q' $.

        There are five possibilities.
        
        \begin{enumerate}
            \item If $ \rho = \lan \ran $ and $ q < q'$, then $ G_i(\lan q \ran \sigma') \in \{q\} \times \chi_i(q) \ll \{q'\} \times \chi_i(q') \ni G_i(\lan q' \ran \tau') $.
            \item If $ \rho = \lan q_0 \ran $ for some $ q_0 \in \Q $ with $ c_i(\lan q_0 \ran) = s_i $, $ q = -\infty $ and $ q' \in \Q $. In this case, $ \sigma' = \lan q_1 \ran $ {for some $ q_1 \in \Q $ with $ h_i^{-1}(q_1) \in c_i(q) = c_i(-\infty) = \alpha_{i}^0 $} and $ G_i(\sigma) = G_i(\lan q_0, -\infty, q_1 \ran) = (q_0, {h_i^{-1}(q_1)}) $, while $ G_i(\tau) = G_i(\lan q_0, q' \ran \tau') = f_{i, q_0} (G_{-i}(\lan q' \ran \tau')) = (q_0, f_i(G_{-i}(\lan q' \ran \tau'))) $.
            Since {$ h_i^{-1}(q_1) \in \alpha^0_i \ll f_i(L_{i}) \ni f_i(G_{-i}(\lan q' \ran \sigma')) $}, we get $ G_i(\sigma) < G_i(\tau) $ as required.
            \item If $ \rho = \lan q_0 \ran $ for some $ q_0 \in \Q $ with $ c_i(\lan q_0 \ran) = s_i $, $ q' = \infty $ and $ q \in \Q $, then the order is preserved by an argument dual to the above case.
            \item If $ \rho = \lan q_0 \ran $ for some $ q_0 \in \Q $ with $ c_i(\lan q_0 \ran) = s_i $, $ q = -\infty $ and $ q'=\infty$. Then, there exist $q'' \in \Q$ with $q<q''<q'$. Then, the conclusion follows together from the Cases (2) and (3) above. 
            
            \item If none of the above is the case, then we are in the inductive step and  $ \rho = \lan q_0 \ran \rho' $ for some nonempty $ \rho' \in \Q^{< \omega} $ with $ \rho'\lan q \ran \sigma', \rho'\lan q' \ran \tau' \in \layer(T_{-i}, n - 1) $ satisfying $ \abs{\rho' \lan q \ran \sigma'} < \abs{\sigma}, \abs{\rho' \lan q' \ran \tau'} < \abs{\tau} $ and $ \rho' \lan q \ran \sigma' \prec \rho' \lan q' \ran \tau' $. Hence, by the inductive hypothesis we have $ G_{-i}(\rho' \lan q \ran \sigma') < G_{-i}(\rho' \lan q' \ran \tau') $. Finally, by the monotonicity of $f_{i,q_0}$ we get $ G_{i}(\sigma) = f_{i, q_0} (G_{-i}(\rho' \lan q \ran \sigma')) < f_{i, q_0}(G_{-i}(\rho' \lan q \ran \tau')) = G_i(\tau) $ as required.
        \end{enumerate}        
\end{proof}

\subsection{\texorpdfstring{$ \Front(T_i) \cong \Xi(F_\plus \cup F_\minus ) $}{Front(T_i) ≅ Ξ(F_+ ∪ F_-)}}
\label{S: 4.2}




\begin{prop}\label{prop: axil iso}
    We have $(\Front (T_i'), c_i \res \Front(T_i') : \Front(T_i') \to F_\plus \cup F_\minus) \cong (\Q, \varphi:  \Q \to F_\plus \cup F_\minus)$, where $ \varphi $ is a dense coloring.
\end{prop}

\begin{proof}
Since $ c_{i}^{-1}(S^0_i) \subseteq\front(T_i') \subseteq\Q^{< \omega}$ by Remark \ref{rem: 1}, we see that $ \front(T_i') $ is nonempty and countable. To see that $\Front(T'_i)$ doesn't have a least element, note that for any $\sigma \in \Front(T'_i)$, there exists $q \in \Q$ with $\lan q \ran \sqsubseteq \sigma$, and then the second clause of  Definition~$\ref{defn: tree}$ guarantees the existence of $q'< q \text{ in } \Q $ with $c_i(\lan q' \ran) \in S^0_i$, so that $ \lan q' \ran \in \Front(T_i') $ with $ \lan q' \ran \prec \sigma $. A dual argument shows that $\Front(T'_i)$ has no greatest element.

Suppose $ \sigma, \tau \in \Front(T_i') $, $ L \in F_\plus \cup F_\minus $ satisfy $ \sigma \prec \tau $. We will show that $ c_i(\zeta) = L $ for some $ \zeta \in \Front(T_i') $ with $ \sigma \prec \zeta \prec \tau $. Since $\sigma \prec \tau$, there are $ \rho \in T_i \text{ and } q, q' \in \ol{\Q} $ satisfying $ \rho \lan q \ran \sqsubseteq \sigma $, $ \rho \lan q' \ran \sqsubseteq \tau $ and $ q < q' $. Then $\rho \in T_i' \setminus \front(T_i') = R^{< \omega} $, and thus $ c_i(\rho) \in \{s_\plus, s_\minus\} $, thanks to Remarks \ref{rem: 1} and \ref{rem: 2}.

If $ L \in S^0_j $ for some $ j \in  \{\plus, \minus\} $, it is possible to choose $p, p' \in \Q $ such that $q < p < q'$ and either $c_i(\tau \lan p \ran) = L$, or $ c_i(\tau \lan p \ran) = s_{-j} $ and $ c_i(\tau \lan p,p'\ran) = L$. Otherwise, $ L \in \{ \alpha^0_j, \delta^0_j \}$ for some $ j \in \{\plus, \minus\} $. In this case, it possible to choose $ p, p' \in \Q $ with $q < p < q'$ such that either $c_i (\tau \lan p \ran) = s_j$, or $ c_i (\tau \lan p \ran) = s_{-j}$ and $c_i (\tau \lan p,p' \ran) = s_j$. Since $ L \not= \mbf 0 $, one of $c_i (\tau \lan p, -\infty \ran)$, $c_i (\tau \lan p,p',-\infty \ran) $, $c_i (\tau \lan p, \infty \ran) $ or $c_i (\tau \lan p,p',\infty \ran)$ equals $ L $. 

We have proved that $\Front(T_i')$ is non-empty, countable, dense, unbounded and $c_i \res \Front(T_i')$ is a dense $ (F_\plus \cup F_\minus) $-coloring. The required isomorphism now follows from Theorem \ref{thm: cantor}.
\end{proof}

\begin{cor}
\label{prop: Front(T) is shuffle}
    $ \Front(T_i) \cong \sum_{r \in \Q} \varphi(r) \cong \Xi(F_\plus \cup F_\minus)) $
\end{cor}
\begin{proof}
    For any $\sigma \in \Front(T_i)$, there exist $\tau \in T_i$ and $q \in \Q$ with $c_i(\tau) \in F_\plus \cup F_\minus $, $q \in c_i(\tau)$ and $\tau \lan q \ran = \sigma$.
    Thus if $\sigma_1 \prec \sigma_2$ in $\Front(T_i)$ with $\sigma_1 = \tau_1 \lan q_1 \ran$ and $\sigma_2 = \tau_2 \lan q_2 \ran$, we have either $\tau_1 \prec \tau_2$ or $\tau_1 = \tau_2$ and $q_1 < q_2$. Using Proposition \ref{prop: axil iso}, fix a colored order isomorphism $ H' : (\Front(T_i'), c_i: \Front(T_i') \to F_\plus \cup F_\minus) \to (\Q, \varphi) $. Then the map $ H : \Front(T_i) \to \Xi(F_\plus \cup F_\minus) $ defined by $ H(\tau\lan q \ran) = (H'(\tau), h_{(-1)^{\abs{\tau}-1}i}(q)) $ is an order isomorphism. 
\end{proof}

Proposition~\ref{prop: Front(T) is L} and Corollary \ref{prop: Front(T) is shuffle} together imply the following.

\begin{cor}
\label{cor: without Aiw}
    $ L_i \setminus A_{i, \omega} \cong \Xi(F_\plus \cup F_\minus) $.
\end{cor}

\subsection{\texorpdfstring{$L_+ \cong L_- $}{L_+ ≅ L_-}}
\label{subsection: final argument}


    



\begin{prop}
    For any $ \sigma \in R^{n} $, $ r\in R^\omega $ with $ \sigma \prec r $, $ A_{i, n}^\sigma \ll A_{i, \omega}^r $.
\end{prop}
\begin{proof}
    Since $ \sigma \prec r $, there exist $ \rho \in R^{< \omega} $, $ q, q' \in R $, such that $ \rho \lan q \ran \sqsubseteq \sigma $, $ \rho \lan q' \ran \sqsubseteq r $ and $ q < q' $. Then Remark \ref{rem: decreasing} yields $ A_{i, n}^\sigma = \rng(f_{-i, \sigma}) \subseteq \rng(f_{-i, \rho\lan q \ran}) = f_{-i, \rho}(f_{(-1)^{n+1}i, q}(L_i)) \ll f_{-i, \rho}(f_{(-1)^{n+1}i, q'}(L_i)) = \rng(f_{-i, \rho \lan q' \ran}) = A_{i, \abs{\rho \lan q' \ran }}^{\rho \lan q' \ran} \supseteq  A_{i, \omega}^r $.  
\end{proof}
\begin{prop}
    Suppose $ r_0, r_1 \in R^{\omega} $ differ at only finitely many places. Then $ A_{i, \omega}^{r_0} \cong A_{i, \omega}^{r_1} $.
\end{prop}
\begin{proof}
    Suppose $ r_0, r_1 \in R^{\omega},  N \in \N $ satisfy $ r_0(n) = r_1(n) $ for every $ n \ge N $,  For each $ k \in \{0, 1\} $, let $ \sigma_k := r_k \res N $. Let $ r \in R^{\omega} $ be defined by $ r(n) := r_0(N + n) = r_1(N + n) $ for each $ n \in \N $. Also set $ \rho_n := r \res n $ for each $ n \in \N $. Consider the isomorphism $ F_{\sigma_0}^{\sigma_1} : A_{i, n}^{\sigma_0} \to A_{i, n}^{\sigma_1} $. Now it suffices to observe that $$ F_{\sigma_0}^{\sigma_1} (A_{i, \omega}^{r_0}) = F_{\sigma_0}^{\sigma_1}(\bigcap_{n \in \N} A_{i, N + n}^{\sigma_0 \rho_n}) = \bigcap_{n \in \N} (F_{\sigma_0}^{\sigma_1}(A_{i, N + n}^{\sigma_0 \rho_n})) = \bigcap_{n \in \N} F_{\sigma_0 \rho_n}^{\sigma_1 \rho_n}(A_{i, N + n}^{\sigma_0 \rho_n}) = \bigcap_{n \in \N} A_{i, N + n}^{\sigma_1 \rho_n} = A_{i, \omega}^{r_1},$$ where the first equality is a consequence of Remark~\ref{rem: decreasing} while the third follows from Remark~\ref{rem: stability}.  
\end{proof}

The above proposition motivates the following construction.

\begin{defn}
    Let $ {\sim} \subseteq R^\omega \times R^\omega $ be the equivalence relation with $ r \sim r' $ if and only if $ r(n) = r'(n) $ for all but finitely many $ n \in \N $.
    For each $ r \in R^\omega $, denote the $ \sim $-equivalence class of $ r $ as $ [r] $. Choose and fix a representative element $ \ol r \in [r] $ in each $\sim$-equivalence class. Also fix an isomorphism $ h_r : A_{i, \omega}^r \to A_{i, \omega}^{\ol r} $ for each $r \in R^\omega$. Finally, define ${F_{i,0}} := \{ A_{i, \omega}^{\ol r} \mid r \in R^\omega\}$
\end{defn}
Fix $ \lan q_{00} \ran \in R $.
\begin{rem}
\label{rem: same colors}
     For any $ r \in R^\omega $ and $ q \in R $, we have $ f_{i, q}(A_{i, \omega}^r) = A_{-i, \omega}^{\lan q \ran r}$. {The bijective assignment $ z([r]) := [\lan q_{00} \ran r] $ induces a bijection $F_{i,0} \to F_{-i,0}$ defined by $ A_{i, \omega}^{\ol r} \mapsto A_{-i, \omega}^{\ol{z(r)}} $.}
\end{rem}

\begin{prop}\label{prop: eq class density}
    For each $ r \in R^\omega $, $ [r] $ is dense in $ \ol\Front(T_i') $.
\end{prop}

\begin{proof}
    It suffices to show that for every $ \sigma, \tau \in \ol{\Front}(T_i') $ with $ \sigma \prec \tau $, there exists some $ r' \in [r] $ with $ \sigma \prec r' \prec \tau $. Let $ \rho \in R^n $, $ q, q' \in \Q $ be such that $ \rho \lan q \ran \sqsubseteq \sigma $ and $ \rho \lan q' \ran \sqsubseteq \tau $. Let $ q'' \in R $ be such that $  q < q'' < q' $ and let $ r' \in R^\omega $ be given by $ r'(n) := r(\abs{\rho} + 1 + n) $. Then $ \rho \lan q'' \ran r' \sim r $ and $ q < \rho \lan q'' \ran r' < q' $.
\end{proof}

Define an extension $C_i:\ol\front(T_i')= R^\omega \cup \front(T_i') \to \ol{F_\plus} \cup \ol{F_\minus} \cup F_{i,0}$ of the coloring $ c_i\res\front(T'_i) $  by setting $ C_i(x) :=\begin{cases} c_i(x)&\mbox{ if }x\in\front(T'_i);\\
        A_{i,\omega}^{\ol x}&\mbox{ if  }x \in R^{\omega}.\end{cases}$
\begin{rem}\label{rem: colour set bijection}
 The coloring map $C_i$ is surjective and the bijection from Remark \ref{rem: same colors} extends to a bijection $ \Gamma : \rng(C_\plus) \to \rng(C_\minus) $ satisfying $ L \cong \Gamma(L) $ for each $ L \in \rng(C_\plus) $.   
\end{rem}

\begin{rem}
\label{rem: Ci}
    Propositions~\ref{prop: axil iso}, \ref{prop: eq class density} and the fact that $ [r] $ is countable for any $ r \in R^\omega $ together imply that $ \{c_i^{-1}(\sigma) \mid \sigma \in \Front(T_i')\} \cup \{[r] \mid r \in R^\omega\} $ is a partition of $ \ol \Front(T_i') $ into continuum-many countable dense $ C_i $-monochromatic suborders.
\end{rem}


    

\begin{prop}
\label{prop: L_i as an ordered sum}
    $ L_i \cong \sum_{x \in \ol \Front(T_i')} C_i(x) .$
\end{prop}

\begin{proof}
Let $ H_i : L_i \to \sum_{x \in \ol \Front(T_i')} C_i(x) $ be defined by

$$H_i(x) := \begin{cases}
        (\rho\lan q_0 \ran, h_{(-1)^{\abs{\rho} + 1} i}^{-1}(q_1)) & \text{ if } x \notin A_{i, \omega}, G_{i}^{-1}(x) = \rho \lan q_0, q_1\ran \text{ for some } \rho \in R^{< \omega}, q_0, q_1 \in \Q ; \\
        (\rho\lan q_0, \pm \infty \ran, h_{(-1)^{\abs{\rho} + 1} i}^{-1}(q_1)) & \text{ if } x \notin A_{i, \omega}, G_{i}^{-1}(x) = \rho \lan q_0, \pm \infty, q_1\ran \text{ for some } \rho \in R^{< \omega}, q_0, q_1 \in \Q ;\\
        (r, h_r(x))&\text{ if }x  \in A_{i, \omega}\text{ and } r \in R^\omega \text{ is unique such that } x \in A_{i, \omega}^r. 
    \end{cases} $$
The map $H_i$ is well-defined thanks to Proposition \ref{prop: Aiw partitioned into intervals} and Remark \ref{rem: leaves}.

We need to show that $ H_i $ is an order isomorphism. Firstly, recall from the proof of Corollary~\ref{cor: without Aiw} that $ H_i \res (L_i \setminus A_{i, \omega}) : L_i \setminus A_{i, \omega} \to \sum_{\sigma \in \Front(T_i')} C_i(\sigma) $ is an order isomorphism. Secondly, observe that Proposition~\ref{prop: Aiw partitioned into intervals} implies that $ H_i \res A_{i, \omega} : A_{i, \omega} \to \sum_{r \in R^\omega} C_i(r) $ is also an order isomorphism. {We will now show these two restrictions are {\emph{compatible}}, i.e., $ H_i(x) < H_i(y) $ for $ x \in L_i \setminus A_{i, \omega} $ and $ y \in A_{i, \omega} $ with $ x < y $. The proof when $x>y$ is dual, and is omitted.}

Since $x\in L_i\setminus A_{i,\omega}$, thanks to Remark \ref{rem: leaves}, $ G^{-1}_i(x) $ is of the form $ \rho \lan q_0, q_1 \ran $ for $\rho\in R^{<\omega}$ and $q_0\in \Q\setminus R$, or $\rho \lan q_0,\pm \infty, q_1 \ran $ for $\rho\lan q_0\ran\in R^{<\omega}$. Using the definition of $ G_i $, we conclude that $ x \in \rng(f_{i, \rho}) = A_{i, n}^\rho $, where $ n := \abs{\rho} $. Since $y\in A_{i,\omega}$,  there is a unique $ r \in R^{\omega} $ such that $ y \in A_{i, \omega}^r $, so that $ H_i(y) = (r, h_r(y)) $. 

We complete the proof by considering the following three possibilities.
    \begin{itemize}
        \item If $\rho \lan q_0 \ran \sqsubseteq r$, then  $G_i^{-1}(x) = \rho \lan q_0, \lambda , q_1 \ran, $  and $   H_i(x) = (\rho \lan q_0, \lambda \ran, h_{j}^{-1}(q_1)) $, for $ j := (-1)^{n + 1}i $ and some $ \lambda \in \{- \infty, \infty\}. $ Let $q' \in \Q$ be such that $ \rho \lan q_0, q' \ran \sqsubseteq r $. We must have $ \lambda = -\infty $, for otherwise by Remark \ref{rem: decreasing}, we would have $ y \in A_{i, \omega}^r \subseteq A_{i, n + 2}^{\rho\lan q_0, q' \ran} = \rng(f_{-i, \rho \lan q_0, q' \ran}) \subseteq \rng(f_{-i, \rho \lan q_0 \ran}) = f_{-i, \rho}(f_{j, q_0}(L_j)) \ll f_{-i, \rho}(\{q_0 \} \times \delta_j^0) \ni f_{-i, \rho}(q_0, h_{j}^{-1}(q_1)) = G_i(\rho \lan q_0, \infty, q_1 \ran) \ni x $, which is a contradiction.

        Since $-\infty< r(n+2)$, we have $ G_i^{-1}(x) = \rho\lan q_0, -\infty, q_1 \ran\prec r$, and hence $ H_i(x) <  H_i(y) $ as required.


        \item If $\rho \not\sqsubseteq r$, then there exist $\sigma \in R^{< \omega}$, $q', q'' \in R$ such that $\sigma \lan q' \ran \sqsubseteq r$ and $\sigma \lan q'' \ran \sqsubseteq \rho$. If $ q' < q'' $, then $ y \in A_{i,\omega}^r \subseteq A_{i, \abs{\sigma \lan q' \ran}}^{\sigma \lan q' \ran} = \rng(f_{-i,\sigma \lan q' \ran}) = f_{-i, \sigma}(f_{j, q'}(L_j)) \ll f_{-i, \sigma}(f_{j, q''}(L_j)) = \rng (f_{-i,\sigma \lan q'' \ran}) = A_{i, \abs{\sigma \lan q'' \ran}}^{\sigma \lan q'' \ran} \ni x$ for $j := (-1)^{\abs{\sigma}+1} i$, contradicting the assumption $ x < y $. Therefore, $ q'' < q' $, i.e., $ \rho \prec r $, and hence $ H_i(x) < H_i(y) $. 
        \item If $\rho \sqsubseteq r$ but $\rho \lan q_0 \ran \not\sqsubseteq r$, then there exists $q' \in R$ such that $\rho \lan q' \ran \sqsubseteq r$. Assume that $q_0 > q'$. Then, $ y \in A_{i,\omega}^r \subseteq A_{i, n+1}^{\rho \lan q' \ran} = \rng(f_{-i,\rho \lan q' \ran}) = f_{-i, \rho}(f_{j, q'}(L_j)) \ll f_{-i, \rho}(f_{j, q_0}(L_j)) = \rng (f_{-i,\rho \lan q_0 \ran}) = A_{i, n+1}^{\sigma \lan q_0 \ran} \ni x$ for $j := (-1)^{\abs{\rho}+1} i$, contradicting the assumption $ x < y $. Therefore, $ q_0 < q' $, i.e., $\rho \prec r$, and hence $ H_i(x) < H_i(y) $.
    \end{itemize}

\end{proof}

\begin{cor}
\label{cor: with Aiw}
 $ L_i \cong \Xi(\rng(C_i))$.
\end{cor}

\begin{proof}
Since $ L_i $ is countable, Proposition~\ref{prop: L_i as an ordered sum} implies that $ C_i(r) = \mbf 0 $ for all but countably many $ r \in R^\omega $. Hence, $ \rng(C_i) $ is countable. Let $ R'_{i} := \{r\in R^\omega\mid C_i(r)\neq\mbf0\}$ and $ Q_i := (\front(T_i') \sqcup R'_i, \prec) $.  Note that $ Q_i $ is a countable dense linear order without end points. Thanks to Remark~\ref{rem: Ci}, $C_i\res Q_i$ is a dense $\rng(C_i)$-coloring. Therefore, $ L_i \cong \sum_{x \in Q_i} C_i(x) \cong {\Xi(\rng(C_i))} $.
 \end{proof}

\begin{proof} (of Theorem \ref{thm: main theorem})
Recall from Remark \ref{rem: colour set bijection} that there is a bijection $\Gamma:\rng(C_+)\to\rng(C_-)$ satisfying $L\cong\Gamma(L)$ for each $L\in \rng(C_+)$. Combining this bijection with the isomorphism from the above corollary, we get  $ L_\plus \cong \Xi({\rng(C_\plus)}) \cong \Xi(\rng(C_\minus)) \cong  L_\minus $.
\end{proof}

\section{Future directions}

In this paper, we used countability arguments twice.

First, in Section~\ref{S: Constructing}, we used the countability of each linear order in {$ F_i $} to fix embeddings $ h_i $ that allowed us construct $ T_i $ as a subset of $ \ol\Q^{< \omega} $. This can be avoided: by first constructing $ T_i' $ (e.g., by removing condition $ (1) $ in Definition~\ref{defn: tree}), and then obtaining $ T_i $ by adding, below each $ \sigma \in \Front(T_i') $, a copy of the potentially uncountable $ c_i(\sigma) $ as the set of its children, with $ \prec $ being defined using the order relations on the $ c_i(\sigma) $'s, we can proceed without cardinality arguments as far as to obtain Proposition~\ref{prop: L_i as an ordered sum}.  

Second, in the proof of Corollary~\ref{cor: with Aiw}, we used the countability of $ L_i $  to conclude that $ C_i(r) = \mbf 0 $ for all but countably many $ r \in R^\omega $. This allowed us to easily obtain the isomorphism $ L_i = \sum_{x \in \ol\Front(T'_i)} C_i(x) \cong \Xi(\rng(C_i)) $ by invoking Theorem~\ref{thm: cantor}. However, in the general case, $ C_i(r) $ may well be nonempty for uncountably many $ r \in R^\omega $.
Nonetheless, the first author conjectures the following generalization of Theorem~\ref{thm: main theorem}.

\begin{conj}

    Let $ L_\plus, L_\minus $ be convex-biembeddable linear orders and suppose that there exist countable sets $ S_\plus^0, S_\minus^0 $ of linear orders such that $L_\plus \cong \Xi(S_\plus^0) $ and $ L_\minus \cong \Xi(S_\minus ^0) $.  Then $ L_\plus \cong L_\minus $.

\end{conj}
\def\cont{\mathfrak{c}}

{The first author also believes that the results of this paper are applicable more generally.}

\begin{ques}

Let $ M $ be a countable model of an $ \aleph_0 $-categorical theory for a purely relational finite language. Using appropriate analogues of convex embeddings and shuffles with respect to suitable coloring functions, is it possible to obtain a CSB-style theorem using techniques similar to those developed in this paper? 


    
    
\end{ques}

\section*{Declarations}
\subsection*{Compliance with ethical standards}
Not applicable

\subsection*{Competing interests}
Not applicable

\subsection*{Funding}
The authors have no relevant financial or non-financial interests to disclose.

\subsection*{Availability of data and materials}
Not applicable

\bibliographystyle{alpha}
\bibliography{main}

\begin{thebibliography}{Myh55}

\bibitem[Esc21]{escardo}
Martin Escardo.
\newblock The {C}antor–{S}chr{\"o}der–{B}ernstein theorem for $ \infty $-groupoids.
\newblock {\em Journal of Homotopy and Related Structures}, 16(3):363--366, September 2021.

\bibitem[Hei80]{heilbrunner1980algorithm}
Stephan Heilbrunner.
\newblock An algorithm for the solution of fixed-point equations for infinite words.
\newblock {\em RAIRO. Informatique th{\'e}orique}, 14(2):131--141, 1980.

\bibitem[LL68]{leonard1968elementary}
Hans L\"auchli and John~J. Leonard.
\newblock On the elementary theory of linear order.
\newblock {\em Journal of Symbolic Logic}, 33(2), 1968.

\bibitem[MK24]{mittal2024exponentiablelinearordersneed}
Mihir Mittal and Amit Kuber.
\newblock Exponentiable linear orders need not be transitive.
\newblock {\em \texttt{arXiv:2406.10532}}, 2024.

\bibitem[Myh55]{myhill1955creative}
John Myhill.
\newblock Creative sets.
\newblock {\em Mathematical Logic Quarterly}, 1(2), 1955.

\bibitem[Ros81]{rosenstein}
Joseph~G. Rosenstein.
\newblock {\em {L}inear orderings}.
\newblock Academic Press New York, 1981.

\bibitem[Sik48]{Sikorski1948}
Roman Sikorski.
\newblock On a generalization of theorems of banach and cantor-bernstein.
\newblock {\em Colloquium Mathematicum}, 1(2):140--144, 1948.

\bibitem[SK24]{SK}
Suyash Srivastava and Amit Kuber.
\newblock Automating the stable rank computation for special biserial algebras.
\newblock {\em \texttt{arXiv:2407.02326}}, 2024.

\bibitem[Sko70]{Skolem1970-SKOLUB}
Thoralf Skolem.
\newblock Logisch-kombinatorische untersuchungen \"{U}ber die erf\"{u}llbarkeit oder bewiesbarkeit mathematischer s\"{a}tze nebst einem theorem \"{U}ber dichte mengen.
\newblock In Th~Skolem and Jens~Erik Fenstad, editors, {\em Selected works in logic}, pages 103--115. Universitetsforlaget, 1970.

\bibitem[Sri08]{srivastava2008course}
Shashi~Mohan Srivastava.
\newblock {\em A course on Borel sets}, volume 180.
\newblock Springer Science \& Business Media, 2008.

\bibitem[SSK23]{SSK}
Suyash Srivastava, Vinit Sinha, and Amit Kuber.
\newblock On the stable radical of the module category for special biserial algebras.
\newblock {\em \texttt{arXiv:{2311.10178}}}, 2023.

\bibitem[Tar49]{tarskicardinal}
Alfred Tarski.
\newblock {\em Cardinal algebras}.
\newblock Oxford University Press, 1949.

\end{thebibliography}

\vspace{0.2in}
\noindent{}Corresponding Author: Suyash Srivastava\\
Department of Mathematics and Statistics\\
Indian Institute of Technology Kanpur\\
Uttar Pradesh, India\\
Email 1: \texttt{suyashsriv20@gmail.com}\\
Email 2: \texttt{ss3151@cam.ac.uk}\\
Phone 1: (+44) 74873 29586\\
Phone 2: (+91) 89296 70686\\


\end{document}